\documentclass{amsart}
\usepackage{amssymb}
\usepackage{amsmath}
\usepackage{amsfonts}

\newtheorem{theorem}{Theorem}
\theoremstyle{plain}

\newtheorem{corollary}{Corollary}

\newtheorem{definition}{Definition}
\newtheorem{example}{Example}

\newtheorem{lemma}{Lemma}

\newtheorem{proposition}{Proposition}
\newtheorem{remark}{Remark}

\numberwithin{equation}{section}
\input{tcilatex}

\begin{document}
\title{SLANT RIEMANNIAN SUBMERSIONS FROM SASAKIAN MANIFOLDS}
\author{I. K\"{u}peli Erken}
\address{Art and Science Faculty,Department of Mathematics, Uludag
University, 16059 Bursa, TURKEY}
\email{iremkupeli@uludag.edu.tr}
\author{C. Murathan}
\address{Art and Science Faculty,Department of Mathematics, Uludag
University, 16059 Bursa, TURKEY}
\email{cengiz@uludag.edu.tr}
\date{10.09.2013}
\subjclass[2000]{Primary 53C25, 53C43, 53C55; Secondary 53D15}
\keywords{Riemannian submersion, Sasakian manifold, Anti-invariant
submersion \ \ \ This paper is supported by Uludag University research
project (KUAP(F)-2012/57).}

\begin{abstract}
We introduce slant Riemannian submersions from Sasakian manifolds onto
Riemannian manifolds. We survey main results of slant Riemannian submersions
defined on Sasakian manifolds. We also give an example of \ such slant
submersions.
\end{abstract}

\maketitle

\section{\textbf{Introduction}}

Let $F$ be a $C^{\infty }$-submersion from a Riemannian manifold ($M,g_{M})$
onto a Riemannian manifold $(N,g_{N}).$ Then according to the conditions on
the map $F:(M,g_{M})\rightarrow (N,g_{N}),$ we have the following
submersions:

semi-Riemannian submersion and Lorentzian submersion \cite{FAL}, Riemannian
submersion (\cite{BO1}, \cite{GRAY}), slant submersion (\cite{CHEN}, \cite%
{SAHIN1}), almost Hermitian submersion \cite{WAT}, contact-complex
submersion \cite{IANUS}, quaternionic submersion \cite{IANUS2}, almost $h$%
-slant submersion and $h$-slant submersion \cite{PARK1}, semi-invariant
submersion \cite{SAHIN2}, $h$-semi-invariant submersion \cite{PARK2}, etc.
As we know, Riemannian submersions are related with physics and have their
applications in the Yang-Mills theory (\cite{BL}, \cite{WATSON}),
Kaluza-Klein theory (\cite{BOL}, \cite{IV}), Supergravity and superstring
theories (\cite{IV2}, \cite{MUS}). In \cite{SAHIN}, Sahin introduced
anti-invariant Riemannian submersions from almost Hermitian manifolds onto
Riemannian manifolds. He also suggested to investigate anti invariant
submersions from almost contact metric manifolds onto Riemannian manifolds
in \cite{SAHIN3}.

So the purpose of the present paper is to study similar problems for slant
Riemannian submersions from Sasakian manifolds to Riemannian manifolds. We
also want to carry anti-invariant submanifolds of Sasakian manifolds to
anti-invariant Riemannian submersion theory and to prove dual results for
submersions. For instance, a slant submanifold of a $K$-contact manifold is
an anti invariant submanifold if and only if $\nabla Q=0$ (see: Proposition
4.1 of \cite{alfonso}). We get similar result as Proposition 4. Thus, it
will be worth the study area which is anti-invariant submersions from almost
contact metric manifolds onto Riemannian manifolds.

The paper is organized as follows: In section 2, we present the basic
information about Riemannian submersions needed for this paper. In section
3, we mention about Sasakian manifolds. In section 4, we give definition of
slant Riemannian submersions and introduce slant Riemannian submersions from
Sasakian manifolds onto Riemannian manifolds. We survey main results of
slant submersions defined on Sasakian manifolds. We also give an example of
slant submersions such that characteristic vector field $\xi $ is vertical.

\section{\textbf{Riemannian Submersions}}

In this section we recall several notions and results which will be needed
throughout the paper.

Let $(M,g_{M})$ be an $m$-dimensional Riemannian manifold , let $(N,g_{N})$
be an $n$-dimensional Riemannian manifold. A Riemannian submersion is a
smooth map $F:M\rightarrow N$ which is onto and satisfies the following
axioms:

$S1$. $F$ has maximal rank.

$S2$. The differential $F_{\ast }$ preserves the lenghts of horizontal
vectors.

The fundamental tensors of a submersion were defined by O'Neill (\cite{BO1},%
\cite{BO2}). They are $(1,2)$-tensors on $M$, given by the formula:%
\begin{eqnarray}
\mathcal{T}(E,F) &=&\mathcal{T}_{E}F=\mathcal{H}\nabla _{\mathcal{V}E}%
\mathcal{V}F+\mathcal{V}\nabla _{\mathcal{V}E}\mathcal{H}F,  \label{AT1} \\
\mathcal{A}(E,F) &=&\mathcal{A}_{E}F=\mathcal{V}\nabla _{\mathcal{H}E}%
\mathcal{H}F+\mathcal{H}\nabla _{\mathcal{H}E}\mathcal{V}F,  \label{AT2}
\end{eqnarray}%
for any vector field $E$ and $F$ on $M.$ Here $\nabla $ denotes the
Levi-Civita connection of $(M,g_{M})$. These tensors are called
integrability tensors for the Riemannian submersions. Note that we denote
the projection morphism on the distributions ker$F_{\ast }$ and (ker$F_{\ast
})^{\perp }$ by $\mathcal{V}$ and $\mathcal{H},$ respectively. The following
\ Lemmas are well known (\cite{BO1},\cite{BO2}).

\begin{lemma}
For any $U,W$ vertical and $X,Y$ horizontal vector fields, the tensor fields 
$\mathcal{T}$, $\mathcal{A}$ satisfy:%
\begin{eqnarray}
i)\mathcal{T}_{U}W &=&\mathcal{T}_{W}U,  \label{TUW} \\
ii)\mathcal{A}_{X}Y &=&-\mathcal{A}_{Y}X=\frac{1}{2}\mathcal{V}\left[ X,Y%
\right] .  \label{TUW2}
\end{eqnarray}
\end{lemma}

It is easy to see that $\mathcal{T}$ $\ $is vertical, $\mathcal{T}_{E}=%
\mathcal{T}_{\mathcal{V}E}$ and $\mathcal{A}$ is horizontal, $\mathcal{A=A}_{%
\mathcal{H}E}$.

For each $q\in N,$ $F^{-1}(q)$ is an $(m-n)$ dimensional submanifold of $M$.
The submanifolds $F^{-1}(q),$ $q\in N,$ are called fibers. A vector field on 
$M$ is called vertical if it is always tangent to fibers. A vector field on $%
M$ is called horizontal if it is always orthogonal to fibers. A vector field 
$X$ on $M$ is called basic if $X$ is horizontal and $F$-related to a vector
field $X$ on $N,$ i. e., $F_{\ast }X_{p}=X_{\ast F(p)}$ for all $p\in M.$

\begin{lemma}
Let $F:(M,g_{M})\rightarrow (N,g_{N})$ be a Riemannian submersion. If $\ X,$ 
$Y$ are basic vector fields on $M$, then:
\end{lemma}

$i)$ $g_{M}(X,Y)=g_{N}(X_{\ast },Y_{\ast })\circ F,$

$ii)$ $\mathcal{H}[X,Y]$ is basic, $F$-related to $[X_{\ast },Y_{\ast }]$,

$iii)$ $\mathcal{H}(\nabla _{X}Y)$ is basic vector field corresponding to $%
\nabla _{X_{\ast }}^{^{\ast }}Y_{\ast }$ where $\nabla ^{\ast }$ is the
connection on $N.$

$iv)$ for any vertical vector field $V$, $[X,V]$ is vertical.

Moreover, if $X$ is basic and $U$ is vertical then $\mathcal{H}(\nabla
_{U}X)=\mathcal{H}(\nabla _{X}U)=\mathcal{A}_{X}U.$ On the other hand, from (%
\ref{AT1}) and (\ref{AT2}) we have%
\begin{eqnarray}
\nabla _{V}W &=&\mathcal{T}_{V}W+\hat{\nabla}_{V}W  \label{1} \\
\nabla _{V}X &=&\mathcal{H\nabla }_{V}X+\mathcal{T}_{V}X  \label{2} \\
\nabla _{X}V &=&\mathcal{A}_{X}V+\mathcal{V}\nabla _{X}V  \label{3} \\
\nabla _{X}Y &=&\mathcal{H\nabla }_{X}Y+\mathcal{A}_{X}Y  \label{4}
\end{eqnarray}%
for $X,Y\in \Gamma ((\ker F_{\ast })^{\bot })$ and $V,W\in \Gamma (\ker
F_{\ast }),$ where $\hat{\nabla}_{V}W=\mathcal{V}\nabla _{V}W.$

Notice that $\mathcal{T}$ \ acts on the fibres as the second fundamental
form of the submersion and restricted to vertical vector fields and it can
be easily seen that $\mathcal{T}=0$ is equivalent to the condition that the
fibres are totally geodesic. A Riemannian submersion is called a Riemannian
submersion with totally geodesic fiber if $\mathcal{T}$ $\ $vanishes
identically. Let $U_{1},...,U_{m-n}$ be an orthonormal frame of $\Gamma
(\ker F_{\ast }).$ Then the horizontal vector field $H$ $=\frac{1}{m-n}%
\dsum\limits_{j=1}^{m-n}\mathcal{T}_{U_{j}}U_{j}$ is called the mean
curvature vector field of the fiber. If \ $H$ $=0$ the Riemannian submersion
is said to be minimal. A Riemannian submersion is called a Riemannian
submersion with totally umbilical fibers if 
\begin{equation}
\mathcal{T}_{U}W=g_{M}(U,W)H  \label{4a}
\end{equation}%
for $U,W\in $ $\Gamma (\ker F_{\ast })$. For any $E\in \Gamma (TM),\mathcal{T%
}_{E\text{ }}$and $\mathcal{A}_{E}$ are skew-symmetric operators on $(\Gamma
(TM),g_{M})$ reversing the horizontal and the vertical distributions. By
Lemma 1 horizontally distribution $\mathcal{H}$ is integrable if and only if
\ $\mathcal{A=}0$. For any $D,E,G\in \Gamma (TM)$ one has%
\begin{equation}
g(\mathcal{T}_{D}E,G)+g(\mathcal{T}_{D}G,E)=0,  \label{4b}
\end{equation}%
\begin{equation}
g(\mathcal{A}_{D}E,G)+g(\mathcal{A}_{D}G,E)=0.  \label{4c}
\end{equation}

We recall the notion of harmonic maps between Riemannian manifolds. Let $%
(M,g_{M})$ and $(N,g_{N})$ be Riemannian manifolds and suppose that $\varphi
:M\rightarrow N$ is a smooth map between them. Then the differential $%
\varphi _{\ast }$ of $\varphi $ can be viewed a section of the bundle $\
Hom(TM,\varphi ^{-1}TN)\rightarrow M,$ where $\varphi ^{-1}TN$ is the
pullback bundle which has fibres $(\varphi ^{-1}TN)_{p}=T_{\varphi (p)}N,$ $%
p\in M.\ Hom(TM,\varphi ^{-1}TN)$ has a connection $\nabla $ induced from
the Levi-Civita connection $\nabla ^{M}$ and the pullback connection. Then
the second fundamental form of $\varphi $ is given by 
\begin{equation}
(\nabla \varphi _{\ast })(X,Y)=\nabla _{X}^{\varphi }\varphi _{\ast
}(Y)-\varphi _{\ast }(\nabla _{X}^{M}Y)  \label{5}
\end{equation}%
for $X,Y\in \Gamma (TM),$ where $\nabla ^{\varphi }$ is the pullback
connection. It is known that the second fundamental form is symmetric. If $%
\varphi $ is a Riemannian submersion it can be easily prove that 
\begin{equation}
(\nabla \varphi _{\ast })(X,Y)=0  \label{5a}
\end{equation}%
for $X,Y\in \Gamma ((\ker F_{\ast })^{\bot })$.A smooth map $\varphi
:(M,g_{M})\rightarrow (N,g_{N})$ is said to be harmonic if $trace(\nabla
\varphi _{\ast })=0.$ On the other hand, the tension field of $\varphi $ is
the section $\tau (\varphi )$ of $\Gamma (\varphi ^{-1}TN)$ defined by%
\begin{equation}
\tau (\varphi )=div\varphi _{\ast }=\sum_{i=1}^{m}(\nabla \varphi _{\ast
})(e_{i},e_{i}),  \label{6}
\end{equation}%
where $\left\{ e_{1},...,e_{m}\right\} $ is the orthonormal frame on $M$.
Then it follows that $\varphi $ is harmonic if and only if $\tau (\varphi
)=0 $, for details, \cite{B}.

\section{\textbf{Sasakian Manifolds}}

A $n$-dimensional differentiable manifold $M$ is said to have an almost
contact structure $(\phi ,\xi ,\eta )$ if it carries a tensor field $\phi $
of type $(1,1)$, a vector field $\xi $ and 1-form $\eta $ on $M$
respectively such that 
\begin{equation}
\phi ^{2}=-I+\eta \otimes \xi ,\text{ \ }\phi \xi =0,\text{ }\eta \circ \phi
=0,\text{ \ \ }\eta (\xi )=1,  \label{phi^2}
\end{equation}%
where $I$ denotes the identity tensor.

The almost contact structure is said to be normal if $N+d\eta \otimes \xi =0$%
, where $N$ is the Nijenhuis tensor of $\phi $. Suppose that a Riemannian
metric tensor $g$ is given in $M$ and satisfies the condition 
\begin{equation}
g(\phi X,\phi Y)=g(X,Y)-\eta (X)\eta (Y),\text{ \ \ }\eta (X)=g(X,\xi ).
\label{metric}
\end{equation}%
Then $(\phi ,\xi ,\eta ,g)$-structure is called an almost contact metric
structure. Define a tensor field $\Phi $ of type $(0,2)$ by $\Phi
(X,Y)=g(\phi X,Y)$. If $d\eta =\Phi $ then an almost contact metric
structure is said to be normal contact metric structure. A normal contact
metric structure is called a Sasakian structure, which satisfies 
\begin{equation}
(\nabla _{X}\phi )Y=g(X,Y)\xi -\eta (Y)X,  \label{Nambla fi}
\end{equation}%
where $\nabla $ denotes the Levi-Civita connection of $g$. For a Sasakian
manifold $M=M^{2n+1}$, it is known that 
\begin{equation}
R(\xi ,X)Y=g(X,Y)\xi -\eta (Y)X,  \label{R(xi,,X)Y}
\end{equation}%
\begin{equation}
S(X,\xi )=2n\eta (X)  \label{S(X,xi)}
\end{equation}%
and 
\begin{equation}
\nabla _{X}\xi =-\phi X.  \label{nablaXxi}
\end{equation}%
\cite{BL2}.

Now we will introduce a well known Sasakian manifold example on $%
%TCIMACRO{\U{211d} }%
%BeginExpansion
\mathbb{R}
%EndExpansion
^{2n+1}.$

\begin{example}[\protect\cite{BL1}]
We consider $%
%TCIMACRO{\U{211d} }%
%BeginExpansion
\mathbb{R}
%EndExpansion
^{2n+1}$ with Cartesian coordinates $(x_{i},y_{i},z)$ $(i=1,...,n)$ and its
usual contact form 
\begin{equation*}
\eta =\frac{1}{2}(dz-\dsum\limits_{i=1}^{n}y_{i}dx_{i}).
\end{equation*}%
The characteristic vector field $\xi $ is given by $2\frac{\partial }{%
\partial z}$ and its Riemannian metric $g$ and tensor field $\phi $ are
given by%
\begin{equation*}
g=\frac{1}{4}(\eta \otimes \eta
+\dsum\limits_{i=1}^{n}((dx_{i})^{2}+(dy_{i})^{2}),\text{ \ }\phi =\left( 
\begin{array}{ccc}
0 & \delta _{ij} & 0 \\ 
-\delta _{ij} & 0 & 0 \\ 
0 & y_{j} & 0%
\end{array}%
\right) \text{, \ }i=1,...,n
\end{equation*}%
This gives a contact metric structure on $%
%TCIMACRO{\U{211d} }%
%BeginExpansion
\mathbb{R}
%EndExpansion
^{2n+1}$. The vector fields $E_{i}=2\frac{\partial }{\partial y_{i}},$ $%
E_{n+i}=2\left( \frac{\partial }{\partial x_{i}}+y_{i}\frac{\partial }{%
\partial z}\right) $, $\xi $ form a $\phi $-basis for the contact metric
structure. On the other hand, it can be shown that $%
%TCIMACRO{\U{211d} }%
%BeginExpansion
\mathbb{R}
%EndExpansion
^{2n+1}(\phi ,\xi ,\eta ,g)$ is a Sasakian manifold.
\end{example}

\section{\textbf{Slant Riemannian submersions}}

\begin{definition}
Let $M(\phi ,\xi ,\eta ,g_{M})$ be a Sasakian manifold and $(N,g_{N})$ be a
Riemannian manifold. A Riemannian submersion $F:M(\phi ,\xi ,\eta
,g_{M})\rightarrow $ $(N,g_{N})$ is said to be slant if for any non zero
vector $X\in \Gamma (\ker F_{\ast })-\{\xi \}$, the angle $\theta (X)$
between $\phi X$ and the space $\ker F_{\ast }$ is a constant (which is
independent of the choice of $p\in M$ and of $X\in \Gamma (\ker F_{\ast
})-\{\xi \}$). The angle $\theta $ is called the slant angle of the slant
submersion. Invariant and anti-invariant submersions are slant submersions
with $\theta =0$ and $\theta =\pi /2$, respectively. A slant submersion
which is not invariant nor anti-invariant is called proper submersion.
\end{definition}

Now we will introduce an example.

\begin{example}
$%
%TCIMACRO{\U{211d} }%
%BeginExpansion
\mathbb{R}
%EndExpansion
^{5}$ has got a Sasakian structure as in Example 1. Let $F:%
%TCIMACRO{\U{211d} }%
%BeginExpansion
\mathbb{R}
%EndExpansion
^{5}\rightarrow 
%TCIMACRO{\U{211d} }%
%BeginExpansion
\mathbb{R}
%EndExpansion
^{2}$ be a map defined by $F(x_{1},x_{2},y_{1},y_{2},z)=(x_{1}-2\sqrt{2}%
x_{2}+y_{1},2x_{1}-2\sqrt{2}x_{2}+y_{1})$. Then, by direct calculations 
\begin{equation*}
\ker F_{\ast }=span\{V_{1}=2E_{1}+\frac{1}{\sqrt{2}}E_{4},V_{2}=E_{2},V_{3}=%
\xi =E_{5}\}
\end{equation*}%
and 
\begin{equation*}
(\ker F_{\ast })^{\bot }=span\{H_{1}=2E_{1}-\frac{1}{\sqrt{2}}E_{4},\text{ }%
H_{2}=E_{3}\}.
\end{equation*}%
Then it is easy to see that $F$ is a Riemannian submersion. Moreover, $\phi
V_{1}=2E_{3}-\frac{1}{\sqrt{2}}E_{2}$ and $\phi V_{2}=E_{4}$ imply that $%
\left\vert g(\phi V_{1},V_{2})\right\vert =\frac{1}{\sqrt{2}}$. So $F$ is a
slant submersion with slant angle $\theta =\frac{\pi }{4}.$
\end{example}

In Example 2, we note that the characteristic vector field $\xi $ is
vertical. If $\xi $ is orthogonal to $\ker F_{\ast }$ we will give following
Theorem.

\begin{theorem}
Let $F$ be a slant Riemannian submersion from a Sasakian manifold $M(\phi
,\xi ,\eta ,g_{M})$ onto a Riemannian manifold $(N,g_{N})$. If $\xi $ is
orthogonal to $\ker F_{\ast }$, then $F$ is anti-invariant.
\end{theorem}

\begin{proof}
By (\ref{nablaXxi}), (\ref{2}), (\ref{4b}) and (\ref{TUW}) we have 
\begin{eqnarray*}
g(\phi U,V) &=&-g(\nabla _{U}\xi ,V)=-g(T_{U}\xi ,V)=g(T_{U}V,\xi ) \\
&=&g(T_{V}U,\xi )=g(U,\phi V)
\end{eqnarray*}%
for any $U,V\in \Gamma (\ker F_{\ast }).$ Using skew symmetry property of $%
\phi $ in the last relation we complete the proof of the Theorem.
\end{proof}

\begin{remark}
We note Lotta \cite{LOTTA} proved that if $M_{1}$ is a submanifold of
contact metric manifold of $\tilde{M}_{1}$ and $\xi $ is orthogonal to $%
M_{1} $, then $M_{1}$ is anti-invariant submanifold. So, our result can be
seen as a submersion version of Lotta's result
\end{remark}

Now, let $F$ be a slant Riemannian submersion from a Sasakian manifold $%
M(\phi ,\xi ,\eta ,g_{M})$ onto a Riemannian manifold $(N,g_{N}).$ Then for
any $U,V\in \Gamma (\ker F_{\ast }),$ we put 
\begin{equation}
\phi U=\psi U+\omega U,  \label{TAN}
\end{equation}%
where $\psi U$ and $\omega U$ are vertical and horizontal components of $%
\phi U$, respectively. Similarly, for any $X\in \Gamma (\ker F_{\ast
})^{\perp }$, we have 
\begin{equation}
\phi X=\mathcal{B}X+\mathcal{C}X,  \label{NOR}
\end{equation}%
where $\mathcal{B}X$ (resp.$\mathcal{C}X),$ is vertical part (resp.
horizontal part) of $\phi X$.

From (\ref{metric}), (\ref{TAN}) and (\ref{NOR}) we obtain

\begin{equation}
g_{M}(\psi U,V)=-g_{M}(U,\psi V)  \label{ANTT}
\end{equation}

and%
\begin{equation}
g_{M}(\omega U,Y)=-g_{M}(U,\mathcal{B}Y).  \label{ANTN}
\end{equation}%
for any $U,V$ $\in \Gamma (\ker F_{\ast })$ and $Y\in \Gamma ((\ker F_{\ast
})^{\bot })$.

Using (\ref{1}), (\ref{TAN}) and (\ref{nablaXxi}) we obtain%
\begin{equation}
\mathcal{T}_{U}\xi =-\omega U,\text{ \ \ }\hat{\nabla}_{U}\xi =-\psi U
\label{CONNECTION}
\end{equation}%
for any $U\in \Gamma (\ker F_{\ast }).$

Now we will give the following proposition for a Riemannian submersion with
two dimensional fibers which is similar to Proposition 3.2. of \cite{AL}.

\begin{proposition}
Let $F$ be a Riemannian submersion from almost contact manifold onto a
Riemannian manifold. If dim($\ker F_{\ast })=2$ and $\xi $\ is vertical then
fibers are anti-invariant.
\end{proposition}

As the proof of the following proposition is similar to slant submanifolds
(see \cite{alfonso}) we don't give its proof.

\begin{proposition}
Let $F$ be a Riemannian submersion from a Sasakian manifold $M(\phi ,\xi
,\eta ,g_{M})$ onto a Riemannian manifold $(N,g_{N})$ such that $\xi \in
\Gamma (\ker F_{\ast })$. Then $F$ is anti-invariant submersion if and only
if $D$ is integrable, where $D=\ker F_{\ast }-\{\xi \}$.
\end{proposition}

\begin{theorem}
Let $M(\phi ,\xi ,\eta ,g_{M})$ be a Sasakian manifold \ of dimension $2m+1$
and $(N,g_{N})$ is a Riemannian manifold of dimension $n$. Let $F:M(\phi
,\xi ,\eta ,g_{M})\rightarrow $ $(N,g_{N})$ be a slant Riemannian
submersion. Then the fibers are not totally umbilical.
\end{theorem}

\begin{proof}
Using (\ref{1}) and (\ref{nablaXxi}) we obtain%
\begin{equation}
\mathcal{T}_{U}\xi =-\omega U  \label{T1}
\end{equation}%
for any $U\in \Gamma (\ker F_{\ast })$. If the fibers are totally umbilical,
then we have $\mathcal{T}_{U}V=g_{M}(U,V)H$ for any vertical vector fields $%
U,V$ where $H$ is the mean curvature vector field of any fibre. Since $%
\mathcal{T}_{\xi }\xi $ $=0$, we have $H=0$, which shows that fibres are
minimal. Hence the fibers are totally geodesic, which is a contradiction to
the fact that $\mathcal{T}_{U}\xi =-\omega U\neq 0$.
\end{proof}

By (\ref{1}), (\ref{2}), (\ref{TAN}) and (\ref{NOR}) we have%
\begin{equation}
(\nabla _{U}\omega )V=\mathcal{C}T_{U}V-\mathcal{T}_{U}\psi V,  \label{W}
\end{equation}%
\begin{equation}
(\nabla _{U}\phi )V=\mathcal{BT}_{U}V-\mathcal{T}_{U}\omega V+R(\xi ,U)V,
\label{F}
\end{equation}%
where%
\begin{eqnarray*}
(\nabla _{U}\omega )V &=&\mathcal{H}\nabla _{U}\omega V-\omega \hat{\nabla}%
_{U}V \\
(\nabla _{U}\psi )V &=&\hat{\nabla}_{U}\psi V-\psi \hat{\nabla}_{U}V,
\end{eqnarray*}%
for $U,V\in \Gamma (\ker F_{\ast }).$Now we will characterize slant
submersion by following Theorem.

\begin{theorem}
Let $F$ be a Riemannian submersion from a Sasakian manifold $M(\phi ,\xi
,\eta ,g_{M})$ onto a Riemannian manifold $(N,g_{N})$ such that $\xi \in
\Gamma (\ker F_{\ast })$. Then, $F$ is a slant Riemannian submersion if and
only if there exist a constant $\lambda \in \lbrack 0,1]$ such that 
\begin{equation}
\psi ^{2}=-\lambda (I-\eta \otimes \xi ).  \label{SLANT}
\end{equation}%
Furthermore, in such case, if $\theta $ is the slant angle of $F$, it
satisfies that $\lambda =\cos ^{2}\theta .$
\end{theorem}

\begin{proof}
Firstly we suppose that $F$ is not an anti-invariant Riemannian submersion.
Then, for $U\in \Gamma (\ker F_{\ast }),$%
\begin{equation}
\cos \theta =\frac{g_{M}(\phi U,\psi U)}{\left\vert \phi U\right\vert
\left\vert \psi U\right\vert }=\frac{\left\vert \psi U\right\vert ^{2}}{%
\left\vert \phi U\right\vert \left\vert \psi U\right\vert }=\frac{\left\vert
\psi U\right\vert }{\left\vert \phi U\right\vert }.  \label{COS1}
\end{equation}%
Since $\phi U\perp \xi ,$ we have $g(\psi U,\xi )=0.$ Now, substituting $U$
by $\psi U$ in (\ref{COS1}) \ and using (\ref{metric}) we obtain%
\begin{equation}
\cos \theta =\frac{\left\vert \psi ^{2}U\right\vert }{\left\vert \phi \psi
U\right\vert }=\frac{\left\vert \psi ^{2}U\right\vert }{\left\vert \psi
U\right\vert }.  \label{COS2}
\end{equation}%
From (\ref{COS1}) and (\ref{COS2}) we have 
\begin{equation}
\left\vert \psi U\right\vert ^{2}=\left\vert \psi ^{2}U\right\vert
\left\vert \phi U\right\vert  \label{COS2A}
\end{equation}%
On the other hand, one can get following 
\begin{eqnarray}
g_{M}(\psi ^{2}U,U) &=&g_{M}(\phi \psi U,U)=-g_{M}(\psi U,\phi U)
\label{COS3} \\
&=&-g_{M}(\psi U,\psi U)=-\left\vert \psi U\right\vert ^{2}.  \notag
\end{eqnarray}%
Using (\ref{COS2A}) and (\ref{COS3}) we get 
\begin{eqnarray}
g_{M}(\psi ^{2}U,U) &=&-\left\vert \psi ^{2}U\right\vert \left\vert \phi
U\right\vert  \notag \\
&=&-\left\vert \psi ^{2}U\right\vert \left\vert \phi ^{2}U\right\vert
\label{COS4}
\end{eqnarray}%
Also, one can easily get 
\begin{equation}
g_{M}(\psi ^{2}U,\phi ^{2}U)=-g_{M}(\psi ^{2}U,U).  \label{COS5}
\end{equation}%
So, by help (\ref{COS4}) and (\ref{COS5}) we obtain $g_{M}(\psi ^{2}U,\phi
^{2}U)=\left\vert \psi ^{2}U\right\vert \left\vert \phi ^{2}U\right\vert $
and it follows that $\psi ^{2}U$ and $\phi ^{2}U$ are colineal, that is $%
\psi ^{2}U=\lambda \phi ^{2}U=-\lambda (I-\eta \otimes \xi ).$ Using the
last relation together with (\ref{COS1}) and (\ref{COS2}) we obtain $\cos
\theta =\sqrt{\lambda }$ is constant and so $F$ is a slant Riemannian
submersion.

If $\ F$ is anti-invariant Riemannian submersion $\phi U$ is normal, $\psi
U=0$ and it is equivalent to $\psi ^{2}U=0.$ In this case $\theta =\frac{\pi 
}{2}$ and so the equation (\ref{COS1}) is again provided.
\end{proof}

By using (\ref{metric}), (\ref{TAN}), (\ref{ANTT}) and (\ref{SLANT}) we have
following Lemma.

\begin{lemma}
Let $F$ be a slant Riemannian submersion from a Sasakian manifold $M(\phi
,\xi ,\eta ,g_{M})$ onto a Riemannian manifold $(N,g_{N})$ with slant angle $%
\theta .$Then the following relations are valid 
\begin{equation}
g_{M}(\psi U,\psi V)=\cos ^{2}\theta (g_{M}(U,V)-\eta (U)\eta (V)),
\label{COS5A}
\end{equation}%
\begin{equation}
g_{M}(\omega U,\omega V)=\sin ^{2}\theta (g_{M}(U,V)-\eta (U)\eta (V))
\label{COS6}
\end{equation}%
for any $U,V\in \Gamma (\ker F_{\ast }).$
\end{lemma}

We denote the complementary orthogonal distribution to $\omega (\ker F_{\ast
})$ in $(\ker F_{\ast })^{\bot }$ by $\mu .$ Then we have%
\begin{equation}
(\ker F_{\ast })^{\bot }=\omega (\ker F_{\ast })\oplus \mu .  \label{A1}
\end{equation}

\begin{lemma}
Let $F$ be a proper slant Riemannian submersion from a Sasakian manifold $%
M(\phi ,\xi ,\eta ,g_{M})$ onto a Riemannian manifold $(N,g_{N})$ then $\mu $
is an invariant distribution of $(\ker F_{\ast })^{\bot },$ under the
endomorphism $\phi $.
\end{lemma}

\begin{proof}
For $X\in \Gamma (\mu ),$ from (\ref{metric}) and (\ref{TAN}), we obtain%
\begin{eqnarray*}
g_{M}(\phi X,\omega V) &=&g_{M}(\phi X,\phi V)-g_{M}(\phi X,\psi V) \\
&=&g_{M}(X,V)-\eta (X)\eta (V)-g_{M}(\phi X,\psi V) \\
&=&-g_{M}(X,\phi \psi V).
\end{eqnarray*}%
Using (\ref{SLANT}) and (\ref{A1}) we have%
\begin{eqnarray*}
g_{M}(\phi X,\omega V) &=&-\cos ^{2}\theta g_{M}(X,V-\eta (V)\xi ) \\
&=&g_{M}(X,\omega \psi V) \\
&=&0.
\end{eqnarray*}%
In a similar way, we have $g_{M}(\phi X,U)=-g_{M}(X,\phi U)=0$ due to $\phi
U\in \Gamma ((\ker F_{\ast })\oplus \omega (\ker F_{\ast }))$ for $X\in
\Gamma (\mu )$ and $U\in \Gamma (\ker F_{\ast }).$Thus the proof of the
lemma is completed.
\end{proof}

By help (\ref{COS6}), we can give following

\begin{corollary}
Let $F$ be a proper slant Riemannian submersion from a Sasakian manifold $%
M^{2m+1}(\phi ,\xi ,\eta ,g_{M})$ onto a Riemannian manifold $(N^{n},g_{N}).$%
Let%
\begin{equation*}
\left\{ e_{1},e_{2},...e_{2m-n},\xi \right\}
\end{equation*}%
be a local orthonormal basis of $(\ker F_{\ast }),$ then $\left\{ \csc
\theta we_{1},\csc \theta we_{2},...,\csc \theta we_{2m-n}\right\} $ is a
local orthonormal basis of $\omega (\ker F_{\ast }).$
\end{corollary}

By using (\ref{A1}) and Corollary 1 one can easily prove the following
Proposition.

\begin{proposition}
Let $F$ be a proper slant Riemannian submersion from a Sasakian manifold $%
M^{2m+1}(\phi ,\xi ,\eta ,g_{M})$ onto a Riemannian manifold $(N^{n},g_{N}).$%
Then dim$(\mu )=2(n-m).$ If $\mu =\left\{ 0\right\} ,$then $n=m.$
\end{proposition}

By (\ref{ANTT}) and (\ref{COS5A}) we have

\begin{lemma}
Let $F$ be a proper slant Riemannian submersion from a Sasakian manifold $%
M^{2m+1}(\phi ,\xi ,\eta ,g_{M})$ onto a Riemannian manifold $(N^{n},g_{N}).$
If $e_{1},e_{2},...e_{k},\xi $ are orthogonal unit vector fields in $(\ker
F_{\ast }),$then%
\begin{equation*}
\left\{ e_{1},\sec \theta \psi e_{1},e_{2},\sec \theta \psi
e_{2},...e_{k},\sec \theta \psi e_{k},\xi \right\}
\end{equation*}%
is a local orthonormal basis of $(\ker F_{\ast }).$ Moreover dim$(\ker
F_{\ast })=2m-n+1=2k+1$ and $\dim N=n=2(m-k).$
\end{lemma}

\begin{lemma}
Let $F$ be a slant Riemannian submersion from a Sasakian manifold $M(\phi
,\xi ,\eta ,g_{M})$ onto a Riemannian manifold $(N,g_{N})$ \ If $\omega $ is
parallel then we have%
\begin{equation}
\mathcal{T}_{\psi U}\psi U=-\cos ^{2}\theta (\mathcal{T}_{U}U+\eta (U)\omega
U)  \label{SEC1}
\end{equation}

\begin{proof}
If $\omega $ is parallel, from (\ref{W}), we obtain $\mathcal{C}T_{U}V=%
\mathcal{T}_{U}\psi V$ for $U,V\in \Gamma (\ker F_{\ast }).$We interchange $%
U $ and $V$ and use (\ref{TUW}) we get 
\begin{equation*}
\mathcal{T}_{U}\psi V=\mathcal{T}_{V}\psi U.
\end{equation*}%
Substituting $V$ by $\psi U$ in the above equation and then using Theorem 3
we get the required formula.
\end{proof}
\end{lemma}

We give a sufficient condition for a slant Riemannian submersion to be
harmonic as an analogue of a slant Riemannian submersion from a Sasakian
manifold onto a Riemannian manifold in \cite{SAHIN1}.

\begin{theorem}
Let $F$ be a slant Riemannian submersion from a Sasakian manifold $M(\phi
,\xi ,\eta ,g_{M})$ onto a Riemannian manifold $(N,g_{N})$ \ If $\omega $ is
parallel then $F$ is a harmonic map.
\end{theorem}

\begin{proof}
From \cite{EJ} we know that $F$ is harmonic if and only if $F$ has minimal
fibres. Thus $F$ is harmonic if and only if $\dsum\limits_{i=1}^{n_{1}}%
\mathcal{T}_{e_{i}}e_{i}=0.$ Thus using the adapted frame for slant
Riemannian submersion and by the help of (\ref{6}) and Lemma 5 we can write 
\begin{equation*}
\tau =-\dsum\limits_{i=1}^{m-\frac{n}{2}}F_{\ast }(\mathcal{T}_{e_{i}}e_{i}+%
\mathcal{T}_{\sec \theta \psi e_{i}}\sec \theta \psi e_{i})-F_{\ast }(%
\mathcal{T}_{\xi }\xi ).
\end{equation*}%
Using $\mathcal{T}_{\xi }\xi =0$ we have 
\begin{equation*}
\tau =-\dsum\limits_{i=1}^{m-\frac{n}{2}}F_{\ast }(\mathcal{T}%
_{e_{i}}e_{i}+\sec ^{2}\theta \mathcal{T}_{\psi e_{i}}\psi e_{i})
\end{equation*}%
By virtue of (\ref{SEC1}) in the above equation, we obtain%
\begin{eqnarray*}
\tau &=&-\dsum\limits_{i=1}^{m-\frac{n}{2}}F_{\ast }(\mathcal{T}%
_{e_{i}}e_{i}+\sec ^{2}\theta (-\cos ^{2}\theta (\mathcal{T}%
_{e_{i}}e_{i}+\eta (e_{i})\omega e_{i}))) \\
&=&-\dsum\limits_{i=1}^{m-\frac{n}{2}}F_{\ast }(\mathcal{T}_{e_{i}}e_{i}-%
\mathcal{T}_{e_{i}}e_{i})=0
\end{eqnarray*}%
So we prove that $F$ is harmonic.
\end{proof}

Now setting $Q=\psi ^{2}$,we define $\nabla Q$ by%
\begin{equation*}
(\nabla _{U}Q)V=\mathcal{V}\nabla _{U}QV-Q\hat{\nabla}_{U}V
\end{equation*}%
for any $U,V\in \Gamma (\ker F_{\ast }).$We give a characterization for a
slant Riemannian submersion from a Sasakian manifold $M(\phi ,\xi ,\eta
,g_{M})$ onto a Riemannian manifold $(N,g_{N})$ by using the value of $%
\nabla Q$.

\begin{proposition}
Let $F$ be a slant Riemannian submersion from a Sasakian manifold $M(\phi
,\xi ,\eta ,g_{M})$ onto a Riemannian manifold $(N,g_{N}).$ Then, $\nabla
Q=0 $ if and only if $F$ is an anti-invariant submersion.
\end{proposition}

\begin{proof}
By using (\ref{SLANT}), 
\begin{equation}
Q\hat{\nabla}_{U}V=-\cos ^{2}\theta (\hat{\nabla}_{U}V-\eta (\hat{\nabla}%
_{U}V)\xi )  \label{Q1}
\end{equation}%
for each $U,V\in \Gamma (\ker F_{\ast }),$ where $\theta $ is slant angle.

On the other hand,%
\begin{equation}
\mathcal{V}(\nabla _{U}QV)=-\cos ^{2}\theta (\hat{\nabla}_{U}V-\eta (\hat{%
\nabla}_{U}V)\xi +g(V,\psi U)\xi +\eta (V)\psi U).  \label{Q2}
\end{equation}%
So, from (\ref{Q1}) and $\nabla Q=0$ if and only if $\cos ^{2}\theta
(g(V,\psi U)\xi +\eta (V)\psi U)=0$ which implies that $\psi U=0$ or $\theta
=\frac{\pi }{2}$. Both the cases verify that $F$ is an anti-invariant
submersion.
\end{proof}

We now investigate the geometry of leaves of ($\ker F_{\ast })^{\perp }$ and 
$\ker F_{\ast }.$

\begin{proposition}
Let $F$ be a slant Riemannian submersion from a Sasakian manifold $M(\phi
,\xi ,\eta ,g_{M})$ onto a Riemannian manifold $(N,g_{N}).$Then the
distribution $(\ker F_{\ast })^{\bot }$ defines a totally geodesic foliation
on $M$ if and only if 
\begin{equation*}
g_{M}(\mathcal{H}\nabla _{X}Y,\omega \psi U)-\sin ^{2}\theta g_{M}(Y,\phi
X)\eta (U)=g_{M}(\mathcal{A}_{X}\mathcal{B}Y,\omega U)+g_{M}(\mathcal{H}%
\nabla _{X}\mathcal{C}Y,\omega U)
\end{equation*}%
for any $X,Y\in \Gamma ((\ker F_{\ast })^{\bot })$ and $U\in \Gamma (\ker
F_{\ast }).$
\end{proposition}

\begin{proof}
From (\ref{Nambla fi}) and (\ref{TAN})we have%
\begin{eqnarray}
g_{M}(\nabla _{X}Y,U) &=&-g_{M}(\phi \nabla _{X}\phi Y,U)+g_{M}(Y,\phi
X)\eta (U)  \label{TOTA} \\
&=&g_{M}(\nabla _{X}\phi Y,\phi U)+g_{M}(Y,\phi X)\eta (U)  \notag \\
&=&g_{M}(\nabla _{X}\phi Y,\psi U)+g_{M}(\nabla _{X}\phi Y,\omega
U)+g_{M}(Y,\phi X)\eta (U).  \notag
\end{eqnarray}%
for any $X,Y\in \Gamma ((\ker F_{\ast })^{\bot })$ and $\ U\in \Gamma (\ker
F_{\ast })$.

Using (\ref{Nambla fi}) and (\ref{TAN}) in (\ref{TOTA}), we obtain 
\begin{eqnarray}
g_{M}(\nabla _{X}Y,U) &=&-g_{M}(\nabla _{X}Y,\psi ^{2}U)-g_{M}(\nabla
_{X}Y,\omega \psi U)  \label{TOT3} \\
&&+g_{M}(Y,\phi X)\eta (U)+g_{M}(\nabla _{X}\phi Y,\omega U).  \notag
\end{eqnarray}%
By (\ref{NOR}) and (\ref{SLANT}) we have%
\begin{eqnarray}
g_{M}(\nabla _{X}Y,U) &=&\cos ^{2}\theta g_{M}(\nabla _{X}Y,U)-\cos
^{2}\theta \eta (U)\eta (\nabla _{X}Y)  \label{TOT4} \\
&&-g_{M}(\nabla _{X}Y,\omega \psi U)+g_{M}(Y,\phi X)\eta (U)  \notag \\
&&+g_{M}(\nabla _{X}\mathcal{B}Y,\omega U)+g_{M}(\nabla _{X}\mathcal{C}%
Y,\omega U)  \notag
\end{eqnarray}%
Using (\ref{3}), (\ref{4}) and (\ref{nablaXxi}) in the last equation we
obtain%
\begin{eqnarray*}
\sin ^{2}\theta g_{M}(\nabla _{X}Y,U) &=&\sin ^{2}\theta g_{M}(Y,\phi X)\eta
(U)-g_{M}(\mathcal{H}\nabla _{X}Y,\omega \psi U) \\
&&+g_{M}(\mathcal{A}_{X}\mathcal{B}Y,\omega U)+g_{M}(\mathcal{H}\nabla _{X}%
\mathcal{C}Y,\omega U)
\end{eqnarray*}%
which prove the theorem.
\end{proof}

\begin{proposition}
Let $F$ be a slant Riemannian submersion from a Sasakian manifold $M(\phi
,\xi ,\eta ,g_{M})$ onto a Riemannian manifold $(N,g_{N}).$If the
distribution $\ker F_{\ast }$ defines a totally geodesic foliation on $M$
then $F$ is an invariant submersion.
\end{proposition}

\begin{proof}
By (\ref{CONNECTION}), if the distribution $\ker F_{\ast }$ defines a
totally geodesic foliation on $M$ then we conclude that $\omega U=0$ for any 
$U\in \Gamma (\ker F_{\ast })$ which shows that $F$ is an invariant
submersion.
\end{proof}

\textbf{OpenProblem:}

Let $F$ be a slant Riemannian submersion from a Sasakian manifold $M(\phi
,\xi ,\eta ,g_{M})$ onto a Riemannian manifold $(N,g_{N})$. In \cite%
{alfonso2},Barrera et.al. define and study the Maslov form of non-invariant
slant submanifolds of $S$-space form $\tilde{M}(c)$. They find conditions
for it to be closed. By similar discussion in \cite{alfonso2} we can define
Maslov form $\Omega H$ of $M$ as the dual form of the vector field $\mathcal{%
B}H$, that is,%
\begin{equation*}
\Omega H(U)=g_{M}(U,\mathcal{B}H)
\end{equation*}

for any $U\in \Gamma (\ker F_{\ast }).$ So.it will be interesting for giving
a chararacterization respect to $\Omega H$ for slant submersions, where $%
H=\dsum\limits_{i=1}^{m-\frac{n}{2}}\mathcal{T}_{e_{i}}e_{i}+\mathcal{T}%
_{\sec \theta \psi e_{i}}\sec \theta \psi e_{i}$ and

$\left\{ e_{1},\sec \theta \psi e_{1},e_{2},\sec \theta \psi
e_{2},...e_{k},\sec \theta \psi e_{k},\xi \right\} $ is a local orthonormal
basis of $(\ker F_{\ast }).$


\begin{thebibliography}{99}
\bibitem{AL} {\small Alegre P., \emph{Slant submanifolds of Lorentzian
Saakian and Para Sasakian manifolds, }Taiwanese Journal of Mathematics., 17,
No. 2, (2012), 629-659.}

\bibitem{B} {\small Baird P., Wood J.C.}, {\small \emph{Harmonic Morphisms
Between Riemannian Manifolds,}} {\small London Mathematical Society
Monographs, 29, Oxford University Press, The Clarendon Press, Oxford, 2003.}

\bibitem{alfonso2} {\small Barrera j., Carriazo A., Fernandez L.M.,
Prieto-Martin A. \emph{The Maslow form in non-invariant slant submanifolds
of S-space-forms, }Ann. Mat. Pura Appl. 4, (2012), 803-818.}

\bibitem{BL1} {\small Blair D.E., \emph{Contact manifolds in Riemannian
geometry}, Lectures Notes in Mathematics \textbf{509}, Springer-Verlag,
Berlin, (1976), 146p}

\bibitem{BL2} {\small Blair D. E., \emph{Riemannian geometry of contact and
symplectic manifolds,} Progress in Mathematics. 203, Birkhauser Boston,
Basel, Berlin, 2002.}

\bibitem{BL} {\small Bourguignon J. P. , Lawson H. B. ,} {\small \emph{%
Stability and isolation phenomena for Yang-mills fields, }Commum. Math.
Phys. 79, (1981), 189-230.}

\bibitem{BOL} {\small Bourguignon J. P. , Lawson H. B. ,} {\small \emph{A
Mathematician's visit to Kaluza-Klein theory,}} {\small Rend. Semin. Mat.
Torino Fasc. Spec. (1989), 143-163.}

\bibitem{alfonso} {\small Cabrerizo J. L., Carriazo A., Fernandez L. M. and
Fernandez M., \emph{Slant submanifolds in Sasakian Manifolds,}} {\small %
Glasgow. Math. 42, (2000), 125-138.}

\bibitem{CHEN} {\small Chen B. Y. ,} {\small \emph{Geometry of slant
submanifolds,} Katholieke Universiteit Leuven, Leuven, 1990.}

\bibitem{EJ} {\small Eells J., Sampson J. H., } {\small \emph{Harmonic
Mappings of Riemannian Manifolds, }Amer. J. Math., 86, (1964), 109-160.}

\bibitem{FAL} {\small Falcitelli M. , Ianus S. , and Pastore A. M. }$,%
{\small \emph{Riemannian}}$ {\small \emph{\ submersions and related topics,}}%
, {\small World Scientific Publishing Co., 2004.}

\bibitem{GRAY} {\small Gray A. ,} {\small \emph{Pseudo-Riemannian almost
product manifolds and submersions,} J. Math. Mech, 16 (1967), 715-737.}

\bibitem{IV} {\small IanusS. , Visinescu M. ,} {\small \emph{\ Kaluza-Klein
theory with scalar fields and generalized Hopf manifolds,}} {\small Class.
Quantum Gravity 4, (1987), 1317-1325.}

\bibitem{IV2} {\small IanusS. , Visinescu M. \emph{\ Space-time
compactification and Riemannian submersions,}} {\small In: Rassias, G.(ed.)
The Mathematical Heritage of C. F.} {\small Gauss, (1991), 358-371, World
Scientific, River Edge.}

\bibitem{IANUS2} {\small Ianus S. , Mazzocco R. , Vilcu G. E. ,} {\small 
\emph{Riemannian submersions from quaternionic manifolds}}, {\small Acta.
Appl. Math. 104, (2008) 83-89.}

\bibitem{IANUS} {\small Ianus S. , Ionescu A. M. , Mazzocco R. , Vilcu G. E.
,} {\small \emph{Riemannian submersions from almost contact metric manifolds,%
}} {\small arXiv: 1102.1570v1 [math. DG].}

\bibitem{LOTTA} {\small Lotta A., \emph{Slant submanifolds in contact
geometry,} Bull. Math. Soc. Roumanie 39, (1996), 183-198.}

\bibitem{YK} {\small Yano K., Kon Masahiro. ,} {\small \emph{Anti-invariant
submanifolds ,}} {\small Marcel Dekker, Inc. New York and Bassel, 1976.}

\bibitem{BO1} {\small O'Neill B., \emph{The fundamental equations of
submersion,} Michigan Math. J. 13, (1966) 459-469.}

\bibitem{BO2} {\small O'Neill B., \emph{Semi-Riemannian geometry with
applications to relativity,} Academic Press, New York-London 1983.}

\bibitem{PARK1} {\small Park K. S. ,} {\small \emph{H--slant submersions}},%
{\small \ Bull. Korean Math. Soc. 49 No. 2, (2012), 329-338.}

\bibitem{PARK2} {\small Park K. S. ,} {\small \emph{H-semi-invariant
submersions,}} {\small Taiwan. J. Math., 16(5), (2012), 1865-1878.}

\bibitem{SAHIN} {\small Sahin B. ,} {\small \emph{Anti-invariant Riemannian
submersions from almost Hermitian manifolds}}, {\small Cent. Eur. J. Math.
8(3), (2010) 437-447.}

\bibitem{SAHIN1} {\small Sahin B. ,} {\small \emph{Slant submersions from
almost Hermitian manifolds}}, {\small Bull. Math. Soc. Sci. Math. Roumanie
Tome 54(102) No. 1, (2011), 93 - 105.}

\bibitem{SAHIN2} {\small Sahin B. , \emph{Semi-invariant submersions from
almost Hermitian manifolds}}, {\small Canad. Math. Bull. , 54, No. 3 (2011)}

\bibitem{SAHIN3} S{\small ahin B. , \emph{Riemannian submersions from almost
Hermitian manifolds}}, {\small Taiwanese Journal of Mathematics., 17, No. 2,
(2012), 629-659.}

\bibitem{WAT} {\small Watson B. ,} {\small \emph{Almost Hermitian
submersions,}} {\small J. Differential Geom. , 11(1), (1976), 147-165.}

\bibitem{WATSON} {\small Watson B. , \emph{G, G}}$^{^{\prime }}$-{\small 
\emph{Riemannian submersions and nonlinear gauge field equations of general
relativity,}} {\small In: Rassias, T. (ed.) Global Analysis -} {\small %
Analysis on manifolds, dedicated M. Morse. Teubner-Texte Math., 57 (1983),
324-349, Teubner, Leipzig.}

\bibitem{MUS} {\small M. T. Mustafa, \emph{\ Applications of harmonic
morphisms to gravity,} J. Math. Phys. , 41(10), (2000), 6918-6929.}
\end{thebibliography}
\end{document}